\documentclass[12pt]{amsart}
\usepackage{amssymb,amsmath,amsfonts,amsthm,nomencl,mathrsfs}
\usepackage[arrow, matrix, curve]{xy}
\usepackage[latin1]{inputenc}
\usepackage{a4wide}

\newcommand{\IC}{\mathbb{C}}
\newcommand{\IR}{\mathbb{R}}

\newcommand{\question}[1]{\leavevmode{\marginpar{\tiny%
$\hbox to 0mm{\hspace*{-0.5mm}$\leftarrow$\hss}%
\vcenter{\vrule depth 0.1mm height 0.1mm width \the\marginparwidth}%
\hbox to 0mm{\hss$\rightarrow$\hspace*{-0.5mm}}$\\\relax\raggedright #1}}}

\newcommand{\loc}{\mathrm{loc}}

\newcommand{\IFF}{\mathcal{F}}

\newcommand{\IN}{\mathbb{N}}

\theoremstyle{plain}            
\newtheorem{theorem}{theorem}[section]

\newtheorem{Corollary}[theorem]{Corollary}
\newtheorem{Theorem}[theorem]{Theorem}

\theoremstyle{definition}       
\newtheorem{Definition}[theorem]{Definition}

\newtheorem{Remark}[theorem]{Remark}

\allowdisplaybreaks[1]

\newcommand{\IMM}{\mathscr{M}}

\newcommand{\Ric}{\mathrm{Ric}}

\begin{document}

\begin{titlepage}

\author{Batu G\"uneysu}


\author{Max von Renesse}

\title{Molecules as metric measure spaces with Kato-bounded Ricci curvature}

\end{titlepage}

\begin{abstract} Set $\Psi:=-\log(\tilde{\Psi})$, with $\tilde{\Psi}>0$ the ground state of an arbitrary molecule with $n$ electrons in the infinite mass limit (neglecting spin/statistics). Let $\Sigma\subset \IR^{3n}$ be the set of singularities of the underlying Coulomb potential. We show that the metric measure space $\IMM$ given by $\IR^{3n}$ with its Euclidean distance and the measure 
$$
\mu(dx)=e^{-2\Psi(x)}dx
$$
has a Bakry-Emery-Ricci tensor which is absolutely bounded by the the function $x\mapsto |x-\Sigma|^{-1}$, which we show to be an element of the Kato class induced by $\IMM$. In addition, it is shown $\IMM$ is stochastically complete, that is, the Brownian motion which is induced by a molecule is nonexplosive, and that the heat semigroup of $\IMM$ has the $L^{\infty}$-to-Lipschitz smoothing property. Our proofs reveal a fundamental connection between the above geometric/probabilistic properties and recently obtained derivative estimates for $e^{\Psi}$ by Fournais/S\o rensen, as well as Aizenman/Simon's Harnack inequality for Schrödinger operators. Moreover, our results suggest to study general metric measure spaces having a Ricci curvature which is synthetically bounded from below/above by a function in the underlying Kato class.
\end{abstract}  

\maketitle

\section{Metric measure spaces}

Ever since the pioneering papers by Sturm \cite{sturm00,sturm} and Lott/Villani \cite{lott}, which based on the earlier results from \cite{max,otto,dario}, metric measure spaces with a Ricci curvature which is bounded below by a constant have been examined in great detail and revealed many deep geometric and analytic results and in particular stability properties. Several equivalent definitions of such a lower bound have been given in the last years, which are typically in the spirit of a (possibly rather technical) convexity assumption on a certain nonlinear functional (like for example the convexity of the entropy functional on Wasserstein space). In the situation of a weighted Riemannian manifold, which is a pair given by Riemannian manifold $M$ and a function $\Phi:M\to\IR$, one canonically gets a metric measure space by taking the geodesic distance on $M$ and the weighted Riemannian volume measure $e^{-2\Phi(x)}\mathrm{vol}(dx)$. In this case, the above convexity assumptions turn out (in their simplest dimension free form) to be equivalent to the lower boundedness of the Bakry-Emery Ricci curvature \cite{bakry}
\begin{align}\label{sqaa}
\mathrm{Ric}_{\Phi}:=\mathrm{Ric}+2\nabla^2\Phi.
\end{align}
The reader may find some of the central results on the geometry and analysis of abstract metric measures with lower bounded Ricci curvature in \cite{ags, ags2,erbar,cavaletti,cavaletti2,bacher,gigli,kuwada,naber} and the references therein. \\
The point we want to make in this note is, that, on the other hand, there exist very natural metric measure spaces whose Ricci curvature is not even locally bounded from below or above by a constant, but by a function which lies in the Kato class of the underlying metric measure space. We believe that our main result, Theorem \ref{main} below, suggests to define and investigate systematically metric measure spaces having a Ricci curvature which is bounded from below and/or above by a Kato function.\vspace{1mm}

As we will exclusively work in the setting of metric measure spaces that arise from perturbing the Lebesgue measure of the standard Euclidean metric measure space, we start by briefly recalling the notions from metric measure spaces that will be relevant for us in this special class.\vspace{2mm}

In the sequel, we denote with $(\cdot,\cdot)$ the Euclidean scalar product and with $|\cdot|$ the associated norm. With $C^{0,1}(\IR^m)$ the space of all globally Lipschitz functions on $\IR^m$, that is, the space of all $f:\IR^m\to\IC$ such that
$$
\sup_{x\neq y}|f(x)-f(y)||x-y|^{-1}<\infty,
$$
we understand $C^{0,1}_\loc(\IR^m)$ to be the space of all $f:\IR^m\to\IC$ with $\varphi f\in C^{0,1}(\IR^m)$ for all $\varphi\in  C^{\infty}_c(\IR^m)$.\vspace{1mm}

Assume we are given a real-valued function $\Phi\in  C^{0,1}_\loc(\IR^m)$ with $\Delta\Phi\in L^2_\loc(\IR^m)$. Note that, in particular, $\nabla\Phi \in L^{\infty}_\loc(\IR^m,\IR^m)$ by Rademacher's theorem.

\begin{Definition} The metric measure space $\IMM_{\Phi}$ is defined by
$$
\IMM_{\Phi}:=(\IR^{m}, |\cdot-\cdot|,\mu_{\Phi}),
$$
where $\mu_{\Phi}$ denotes the measure $\mu_{\Phi}(dx):=e^{-2\Phi} dx$. 
\end{Definition}

In view of (\ref{sqaa}), we define the Bakry-Emery Ricci tensor $\Ric_{\Phi}$ of $\IMM_{\Phi}$ by 
$$
\Ric_{\Phi}:= 2\nabla^2 \Phi\in L^2_\loc(\IR^{m},\mathrm{Mat}_{m\times m}(\IR)\big).
$$
Note that the asserted local square integrability follows from $\Phi,\Delta\Phi \in L^2_\loc(\IR^{m})$ and the Calderon-Zygmund inequality. The Cheeger energy form $Q_{\Phi}$ of $\IMM_{\Phi}$ in the Hilbert space $L^2_{\Phi}(\IR^{m})$ is given by 
$$
Q_{\Phi}(f)=\frac{1}{2}\int |\nabla f|^2  d\mu_{\Phi}
$$
with domain of definition $W^{1,2}_{\Phi}(\IR^{m})$ given by all $f\in L^2_{\Phi}(\IR^{m})$ with $\nabla f\in L^2_{\Phi}(\IR^{m},\IC^{m})$, where for $q\in [1,\infty)$ the Banach space $L^q_{\Phi}(\IR^{m})$ is defined to be space of the equivalence classes\footnote{with respect to $dx$ or equivalently $d\mu_{\Phi}$; note that $L^{\infty}_{\Phi}(\IR^{m})=L^{\infty}(\IR^{m})$.} Borel functions $f$ on $\IR^{m}$ with
$$
\int |f|^q d\mu_{\Phi}<\infty,
$$
and likewise for vector-valued functions. The \emph{weighted Laplacian} $\Delta_{\Phi}$ on $\IMM_{\Phi}$ is given by
$$
\Delta_{\Phi}=\Delta - 2(\nabla \Phi, \cdot)
$$
and its domain of definition is $W^{2,2}_{\Phi}(\IR^{m})$, the space of all $f\in L^2_{\Phi}(\IR^{m})$ with $\Delta_{\Phi} f\in L^2_{\Phi}(\IR^{m})$, and an integration by parts shows that $-\Delta_{\Phi}/2\geq 0$ is the self-adjoint operator induced by $Q_{\Phi}$. As in the proof of Theorem 11.5 in \cite{gri} one finds that the operator $-\Delta_{\Phi}/2$ is essentially self-adjoint on $C^{\infty}_c(\IR^m)$.\\
The space-time function $(t,x)\mapsto e^{\frac{t}{2}\Delta_{\Phi}}f(x)$ has a jointly continuous version for all $f\in L^2_{\Phi}(\IR^m)$, and there exists a uniquely determined a jointly continuous map
$$
(0,\infty)\times \IR^m\times \IR^m \ni (t,x,y)\longmapsto e^{\frac{t}{2}\Delta_{\Phi}}(x,y)\in [0,\infty)
$$
which for all $f\in L^2_{\Phi}(\IR^m)$, $t>0$, $x\in \IR^m$ satisfies 
	\begin{align}\label{diri}
e^{\frac{t}{2}\Delta_{\Phi}}f(x) =\int e^{\frac{t}{2}\Delta_{\Phi}}(x,y)f(y) d\mu_{\Phi}(y). 
\end{align}
This integral kernel has the following properties for all $s,t>0$, $x,y\in \IR^m$:
\begin{align*}
&e^{\frac{t}{2}\Delta_{\Phi}}(y,x)= e^{\frac{t}{2}\Delta_{\Phi}}(x,y),\\
	&e^{\frac{t+s}{2}\Delta_{\Phi}}(x,y) =\int e^{\frac{t}{2}\Delta_{\Phi}}(x,y)e^{\frac{s}{2}\Delta_{\Phi}}(y,z)d\mu_{\Phi}(z),\\
	&\int e^{\frac{s}{2}\Delta_{\Phi}}(x,y)d\mu_{\Phi}(y)\leq 1.
	\end{align*}
In particular, one can extend the definition $e^{\frac{t}{2}\Delta_{\Phi}}f(x)$ to $f\in \bigcup_{q\in [1,\infty]} L^q_{\Phi}(\IR^m)$ or all nonnegative Borel $f$'s using formula (\ref{diri}).

\begin{Definition} The Kato class $\mathcal{K}_{\Phi}(\IR^{m})$ of $\IMM_{\Phi}$ is defined by all Borel functions $v$ on $\IR^{m}$ satisfying
$$
\lim_{t\to 0}\sup_{x\in \IR^{m}}\int^t_0\int e^{s\Delta_{\Phi}}(x,y) |v(y)| d\mu_{\Phi}(y) ds=\lim_{t\to 0}\sup_{x\in \IR^{m}}\int^t_0 e^{s\Delta_{\Phi}}|v(x)|  ds=0. 
$$
\end{Definition}

In view of 
$$
\int e^{\frac{s}{2}\Delta_{\Phi}}(x,y)d\mu_{\Phi}(y)\leq 1\quad\text{ for all $s>0$, $x\in \IR^m$,}
$$
one trivially has $L^{\infty}(\IR^m)\subset \mathcal{K}_{\Phi}(\IR^{m})$, while inclusions of the type $L^{q}_{\Phi}(\IR^m)\subset \mathcal{K}_{\Phi}(\IR^{m})$ will in general depend on the geometry induced by $\Phi$. We refer the reader to \cite{aizen, Simon, peter, kt,Batu} for several abstract and Riemannian results on the Kato class. 

\begin{Remark}\label{wspoa} A simple result \cite{Simon, aizen} in this context for $\Phi=1$ is that 
$$
L^{q}(\IR^m)\subset \mathcal{K}(\IR^{m})\quad\text{for all $q>m/2$ if $m\geq 3$},
$$
and that if $v\in \mathcal{K}(\IR^{m})$ and if $T:\IR^{m'}\to \IR^m$ is a surjective linear map, then $v\circ T\in \mathcal{K}(\IR^{m'})$. 
\end{Remark}

Let $(\Omega,\IFF_*,X,(P^x_{\Phi})_{x\in\IR^{m}})$ be the canonical $\Delta_{\Phi}/2$-diffusion, which in general lives on the space of continuous paths $\Omega$ with values in one-point compactification of $\IR^{m}$. In particular, $\IFF_*$ denotes the canonical filtration of $\Omega$ which is generated by the coordinate process $X_t(\omega)=\omega(t)$. For every Borel set $A\subset \IR^m$, $t>0$, $x\in \IR^m$ one has the defining relation
\begin{align*}
P^x_{\Phi}\{A,t<\zeta\} = \int_A e^{\frac{t}{2}\Delta_{\Phi}}(x,y) d\mu_{\Phi}(y),\quad \text{with $\zeta:=\inf\{t\geq 0: X_t=\infty\}:\Omega\longrightarrow [0,\infty]$}
\end{align*}
the canonical explosion time.

\begin{Definition} The metric measure space $\IMM_{\Phi}$ is called \emph{stochastically complete}, if one has $P^x_{\Phi}\{t<\zeta\}=1$ for all $t>0$, $x\in \IR^m$.
\end{Definition}

Note that for $\Phi=1$ the above diffusion is nothing but Euclidean Brownian motion. In particular, $P^x_{\Phi}$ is obtained as the law of the solution $X^{\Phi}(x)$ of the Ito equation
\begin{align}\label{drift}
d X^{\Phi}_t(x)= - \nabla \Phi(X^{\Phi}_t(x)) dt + dW_t,\quad X^{\Phi}_0(x)=x,
\end{align}
where $W$ is a Euclidean Brownian motion.

\section{The metric measure space of a molecule}

Given 
$$
n,m\in\IN,\quad R=(\mathbf{R}_1,\dots,\mathbf{R}_m)\in \IR^{3m},\quad Z=(Z_1,\dots,Z_m)\in \IN^{3m}, 
$$
and the potential $V:\IR^{3n}\to\IR$ 
$$
V(\mathbf{x}_1, \dots,\mathbf{x}_n):=-\sum_{j=1}^m\sum_{i=1}^n  \frac{Z_j}{|\mathbf{x}_i-\mathbf{R}_j|}+ \sum_{1\leq i<j\leq N} \frac{1}{|\mathbf{x}_i-\mathbf{x}_j|},
$$
consider the Hamilton operator $-\Delta/2+V$ in $L^2(\IR^{3n})$ of a molecule with $n$ electrons and $m$ nuclei, where in the infinite mass limit the the $j$-th nucleus is considered to be fixed in $\mathbf{R}_j$. In addition, we have ignored spin/statistics and we have set the elementary charge equal to $1$. This operator is essentially self-adjoint on $C^{\infty}_c(\IR^{3n})$ \cite{Kato} and its domain of definition is $W^{2,2}(\IR^{3n})$. Let $\lambda>0$ be the corresponding ground state energy and $0<\tilde{\Psi}\in W^{2,2}(\IR^{3n})$ the ground state, so 
$$
(-\Delta+V) \tilde{\Psi}=\lambda\tilde{\Psi}.
$$
We know from \cite{Kato} that  $\tilde{\Psi}\in C^{0,1}(\IR^{3n})$. In fact, there exists a constant $C>0$ such that 
$$
\text{  $|\nabla \tilde{\Psi}|\leq C$ on $\IR^{3n}\setminus \Sigma$},
$$
where
$$
\Sigma:= \Big\{x\in\IR^{3n}: \prod^m_{j=1}\prod^n_{i=1}|\mathbf{x}_i-\mathbf{R}_j|\prod_{1\leq i<j\leq N}|\mathbf{x}_i-\mathbf{x}_j|=0  \Big\} 
$$
 is the set of singularities of the underlying Coulomb potential, a closed set of measure zero. Of course, by local elliptic regularity, $\tilde{\Psi}\in C^{\infty}(\IR^{3n}\setminus\Sigma)$.

\begin{Definition} With $\Psi:=- \log(\tilde{\Psi})$, the metric measure space 
$$
\IMM_{\Psi}=(\IR^{3n}, |\cdot-\cdot|,\mu_{\Psi}),
$$
is called \emph{a molecular metric measure space}.
\end{Definition}

\begin{Remark}\label{swsa33} 1. Clearly $\Psi\in  C^{0,1}_\loc(\IR^{3n})$ and it follows from the formula
$$
\Delta \Psi= - 2 (1/\tilde{\Psi})\Delta\tilde{\Psi}-2 (1/\tilde{\Psi}^2)\sum^{3n}_{j=1}(\partial_i\tilde{\Psi})^2
$$
in combination with $|\partial_i\tilde{\Psi}|\leq C$, the continuity of $\tilde{\Psi}$ and $\Delta\tilde{\Psi}\in L^2(\IR^{3n})$ that $\Delta\Psi\in L^2_\loc(\IR^{3n})$, as required for the theory from the previous section.\\
2. The operator $-\Delta/2+V$ in $L^2(\IR^{3n})$ is unitarily equivalent to the operator $-\Delta_{\Psi}/2+\lambda$ in $L^2_{\Psi}(\IR^{3n})$
via
$$
L^2_{\Psi}(\IR^{3n}) \longrightarrow L^2(\IR^{3n})
,\quad U_{\Psi}  f(x)= e^{-\Psi(x)}f(x).
$$
Indeed, as both operators are essentially self-adjoint on $C^{\infty}_c(\IR^{3m})$ in their respective Hilbert spaces, it suffices to check
$$
U_{\Psi}\left(-\Delta_{\Psi}/2+\lambda\right)f= (-\Delta/2+V)U_{\Psi}f\quad\text{ for all $f\in C^{\infty}_c(\IR^{3m})$,}
$$
which is a standard calculation.\\
3. It is well-known \cite{Simon,aizen} that $V\in \mathcal{K}(\IR^{3n})$ (in fact, this is a simple consequence of Remark \ref{wspoa}).
\end{Remark}

Here comes our main result:

\begin{Theorem}\label{main} a) There exist constants $A_1,A_2>0$ such that 
$$
\big|(\Ric_{\Psi}(x)v\big|\leq  \big(A_1+  A_2|x-\Sigma_{}|^{-1}\big)|v|,\quad\text{for all $(x,v)\in (\IR^{3n}\setminus \Sigma)\times \IR^{3n}$}.
$$
b) One has $|\bullet-\Sigma|^{-1}\in \mathcal{K}_{\Psi}(\IR^{3n})$.\\
c) $\IMM_{\Psi}$ is stochastically complete.\\
d) One has the following Lipschitz smoothing property,
$$
|e^{\frac{t}{2}\Delta_{\Psi}}f(x)-e^{\frac{t}{2}\Delta_{\Psi}}f(y)|\leq t^{-1/2} e^{C_2t} |x-y|\quad\text{for all $t>0$, $f\in L^{\infty}(\IR^{3n})$, $x,y\in \IR^{3n}$.}
$$
\end{Theorem}

\begin{proof} Before we come to the actual proof of the statements, we record some central results for atomic Schrödinger operators:
\begin{itemize}
\item[i)] As $V\in\mathcal{K}(\IR^{3m})$, Aizenman/Simon's Harnack inequality \cite{aizen} for positive eigenfunctions of Schrödinger operators with Kato potentials states the existence of a constant $C$ such that 
\begin{align*}
\sup_{\{y\in\IR^{3n}:|y-x|<1/2\}} \tilde{\Psi}(y)\leq C \inf_{\{y\in\IR^{3n}|y-x|<1/2\}} \tilde{\Psi}(y)\quad\text{ for all $x\in \IR^{3n}$}.
\end{align*}
\item[ii)] Recently established estimates by Fournais/S\o rensen (see also \cite{hof} for $\alpha=1$) state that for every multi-index $\alpha$ there exists \cite{fournais} a constant $c_{\alpha}$ with
\begin{align*}
|\partial^{\alpha}\tilde{\Psi}(x)|\leq c_{\alpha} \min(1,|x-\Sigma|)^{1-|\alpha|} \sup_{\{y\in\IR^{3n}:|y-x|<1/2\}}  \tilde{\Psi}(y)  \quad\text{for all $x\in \IR^{3n}\setminus \Sigma$}.
\end{align*}
\item[iii)] As a consequence of i), ii) we get that for every multi-index $\alpha$ one has
\begin{align*}
\frac{|\partial^{\alpha}\tilde{\Psi}(x)|}{\tilde{\Psi}(x)}&\leq c_{\alpha} \min(1,|x-\Sigma|)^{1-|\alpha|} \frac{\sup_{\{y\in\IR^{3n}:|y-x|<1/2\}}  \tilde{\Psi}(y) }{\inf_{\{y\in\IR^{3n}:|y-x|<1/2\}}  \tilde{\Psi}(y) }\\
& \leq c'_{\alpha}C \min(1,|x-\Sigma|)^{1-\alpha},\quad\text{for all $x\in \IR^{3n}\setminus \Sigma$.}
\end{align*}
\item[iv)] For all $q\in [1,\infty)$ the process
$$
\quad \exp\left(- \int^{\cdot}_0 (\nabla (q\Psi)(X_r),dX_r)-\frac{1}{2}\int^{\cdot}_0 |\nabla (q\Psi)(X_r)|^2dr  \right)
$$
is a martingale under $P^x$, where $\int^{\cdot}_0 (\nabla (q\Psi)(X_r),dX_r)$ denotes the Ito-integral
$$
\int^{\cdot}_0 (\nabla (q\Psi)(X_r),dX_r)=q\sum^{3n}_{j=1}\int^{\cdot}_0 (\nabla \Psi)_j(X_r)dX^j_r.
$$
Indeed, by Novikov's theorem it suffices to show that for all $t>0$ one has
\begin{align}\label{khash}
E^x\left[\exp\left( \frac{1}{2}\int^{t}_0 |\nabla (q\Psi)(X_r)|^2dr  \right) \right]<\infty,
\end{align}
which trivially follows from 
\begin{align}\label{opssa}
|\partial_{i}\Psi(x)|=\frac{|\partial_{i}\tilde{\Psi}(x)|}{\tilde{\Psi}(x)} \leq C',
\end{align}
by inequality iii).\\
\item[v)] Girsanov's theorem and uniqueness in law for the solution of the corresponding SDE imply that for all $t>0$, $x\in\IR^{3n}$ we have
$$
dP^x_{\Psi}|_{\IFF_t\cap\{t<\zeta\}}= \exp\left( -\int^t_0 (\nabla \Psi(X_s),dX_s)-\frac{1}{2}\int^t_0 |\nabla \Psi(X_s)|^2ds  \right)d P^x.
$$
\end{itemize}

a) One calculates
$$
\big(\mathrm{Ric}_{\Psi}\big)_{ij}= - \left \{2 (1/\tilde{\Psi})\partial_i\partial_j\tilde{\Psi}-2 (1/\tilde{\Psi}^2)\partial_i\tilde{\Psi}\cdot\partial_j\tilde{\Psi}\right\}\quad\text{ in $\IR^{3n}\setminus\Sigma$.}
$$
In view of the above inequality iii), the absolute value of the first summand can be controlled by $\leq C|\cdot-\Sigma|^{-1}$, while the absolute value of the second summand can be estimated by a constant, again using iii) above.\\
b) Using v), for all $q\in (1,\infty)$ with $q^*\in (1,\infty)$ its Hölder dual, and all $t\geq s\geq 0$, and all $x$,
\begin{align*}
&=\int e^{\frac{s}{2}\Delta_{\Psi}}(x,y) |y-\Sigma|^{-1} d\mu(y) \\
&=E^x_{\Psi}\left[1_{\{s<\zeta\}}|X_s-\Sigma|^{-1}\right]\\
&=E^x\left[\exp\left( -\int^{t}_0 (\nabla \Psi(X_r),dX_r)-\frac{1}{2}\int^{s}_0 |\nabla \Psi(X_r)|^2dr  \right) |X_s-\Sigma|^{-1}\right]\\
&\leq E^x\left[\exp\left(- q^*\int^{s}_0 (\nabla \Psi(X_r),dX_r)-\frac{q^*}{2}\int^{s}_0 |\nabla (\Psi(X_r)|^2dr  \right)\right]^{1/q}\\
&\quad\times E^x\left[ |X_s-\Sigma|^{-q}\right]^{1/q}\\
&\leq e^{C_q t}E^x\left[\exp\left( \int^{s}_0 (\nabla (q^*\Psi)(X_r),dX_r)  -\frac{1}{2}\int^{s}_0 |\nabla (q^*\Psi(X_r)|^2dr  \right)\right]^{1/q}\\
&\quad\times E^x\left[ |X_s-\Sigma|^{-q}\right]^{1/q},
\end{align*}
where we have estimated as follows:
\begin{align*}
&-q^*\int^{s}_0 (\nabla \Psi(X_r),dX_r)-\frac{q^*}{2}\int^{s}_0 |\nabla (\Psi(X_r)|^2dr\\
&\leq -q^*\int^{s}_0 (\nabla \Psi(X_r),dX_r)\\
&= -\int^{s}_0 (\nabla (q^*\Psi)(X_r),dX_r)-\frac{1}{2}\int^{s}_0 |\nabla (q^*\Psi(X_r)|^2dr+\frac{1}{2}\int^{s}_0 |\nabla (q^*\Psi(X_r)|^2dr\\
&\leq - \int^{s}_0 (\nabla (q^*\Psi)(X_r),dX_r)-\frac{1}{2}\int^{s}_0 |\nabla (q^*\Psi(X_r)|^2dr+C'_{q}s,\\
\end{align*}
using (\ref{opssa}). By iv) we have 
$$
E^x\left[\exp\left( -\int^{s}_0 (\nabla (q^*\Psi)(X_r),dX_r)  -\frac{1}{2}\int^{s}_0 |\nabla (q^*\Psi(X_r)|^2dr  \right)\right]=1,
$$
and by Jenßen's inequality, for all $t\leq 1$,
$$
\int^t_0E^x\left[ |X_s-\Sigma|^{-q}\right]^{1/q}ds\leq \left(\int^t_0E^x\left[ |X_s-\Sigma|^{-q}\right]ds\right)^{1/q}.
$$
In in order to estimate this, we record that for all $y=(\mathbf{y}_1,\dots,\mathbf{y}_n)\in \IR^{3n}$ one has
\begin{align*}
&|y-\Sigma|^{-q}\\
&= \min\Big\{ |\mathbf{y}_i-\mathbf{R}_j|,\frac{1}{\sqrt{2}}|\mathbf{y}_k-\mathbf{y}_l|: i,k,l\in\{1,\dotsm,n\},k<l, j\in\{1,\dots,m\}  \Big\}^{-q}\\
&\leq \sum_{j=1}^m\sum_{i=1}^n  |\mathbf{y}_i-\mathbf{R}_j|^{-q}+ \sum_{1\leq k<l\leq n} \left(\frac{1}{\sqrt{2}}|\mathbf{y}_k-\mathbf{y}_l|\right)^{-q},
\end{align*}
and the latter function of $y$ is in $\mathcal{K}(\IR^{3n})$ for some $q>1$ by Remark \ref{wspoa} (just note that for all $\mathbf{R}\in \IR^3$ the function 
$$
\mathbf{y}\mapsto |\mathbf{y}-\mathbf{R}|^{-3/2}=1_{\{\mathbf{z}:|\mathbf{z}-\mathbf{R}|\leq 1\}}(\mathbf{y})|\mathbf{y}-\mathbf{R}|^{-3/2}+1_{\{\mathbf{z}:|\mathbf{z}-\mathbf{R}|>1\}}(\mathbf{y})|\mathbf{y}-\mathbf{R}|^{-3/2}
$$
is an element of $L^{3/2}(\IR^3)+L^{\infty}(\IR^3)$). Thus for such a $q>1$ we arrive at
$$
\sup_{x\in \IR^{3n}}\Big(\int^t_0E^x\left[ |X_s-\Sigma|^{-q}\right]ds\Big)^{1/q}=\Big(\sup_{x\in \IR^{3n}}\int^t_0E^x\left[ |X_s-\Sigma|^{-q}\right]ds\Big)^{1/q}\to 0
$$
as $t\to 0+$, which completes the proof of a).\\
c) Using v), iv) we have
\begin{align*}  
&\int e^{\frac{t}{2}\Delta_{\Psi}}(x,y) d\mu(y)\\
& =P^x_{\Psi}\{t<\zeta\}\\
&= E^x\left[\exp\left( -\int^{t}_0 (\nabla \Psi(X_s),dX_s)-\frac{1}{2}\int^{t}_0 |\nabla \Psi(X_s)|^2ds  \right) \right]|_{t=0}\\
&=1,
\end{align*}
completing the proof.\\
d) We proceed by using the Bismut-Elworthy-Li representation of the derivative of a family of smoothened versions of the semigroup $e^{\frac{t}{2}\Delta_{\Psi}}$. To this aim fix $\epsilon >0$ and assume for a moment that we replace $\Psi$ by a mollified version $\Psi\epsilon:= e^{\frac \epsilon 2 \Delta} \Psi$, $e^{\frac t 2 \Delta}$ denotes the standard Euclidean heat semigroup. Let $\mathrm{Ric}_{\Psi^\epsilon}$ and $e^{\frac{t}{2}\Delta_{\Psi^\epsilon}}$ be the corresponding Ricci tensor and corresponding semigroup respectively which are associated to this choice of $\Phi := \Psi^\epsilon$. Then $\Psi^\epsilon$ is smooth, $\|\nabla \Psi_ \epsilon\|_\infty \leq \|\nabla \Psi\|_\infty$ and  $|\mathrm{Ric}_{\Psi^\epsilon}| \leq k^{\epsilon}$ where $k^\epsilon:= e^{\frac \epsilon 2 \Delta} k$ with $k(\bullet):= |\bullet-\Sigma|^{-1}$. Fix $x\in\IR^{3n}$, $t>0$. In view of (\ref{drift}), let $X^{\epsilon}$ be the solution of
$$
d X^{\epsilon}_s= - \nabla \Psi^{\epsilon}(X^{\epsilon}_s) ds + dW_s,\quad X^{\epsilon}_0=x
$$
on some filtered probability space whose expectation is denoted by $E[\bullet]$, where $W$ is a Euclidean Brownian motion in $\IR^{3n}$. Following the presentation in \cite{thth}, let the $\mathrm{End}((\IR^{3n})^*)$-valued process $Q^{\epsilon}$ be defined as the solution of the ODE
$$
d Q^{\epsilon}_s = - \frac 1 2 Q^{\epsilon}_s\mathrm{Ric}_{\Psi^{\epsilon}}(X^{\epsilon}_s)ds,\quad Q^{\epsilon}_0=1, 
$$
which admits the bound
\[ |Q^{\epsilon}_\bullet| \leq e^{ \frac 1 2 \int_0^\bullet k^{\epsilon}(X^{\epsilon}_s) ds} \] 
due to $|\mathrm{Ric}_{\Psi^\epsilon}|\leq k^{\epsilon}$ and Gronwall's inequality. \\
Assume for the moment that $f$ is smooth and compactly supported. Combining the parabolic PDE solved by the derivative $d e^{\frac{t-\bullet}{2}\Delta_{\Psi^{\epsilon}}}$ of the semigroup with Ito's formula for $X^\epsilon$ we see that process
$$
N^{\epsilon}_\bullet :=Q^\epsilon_\bullet d e^{\frac{t-\bullet}{2}\Delta_{\Psi^{\epsilon}}}f(X^{\epsilon}_\bullet)
$$  
with values in $(\IR^{3n})^*$ is a local martingale. Thus, if $\ell$ is an adapted $\IR^{3n}$-valued process with paths in the Cameron-Martin space, then
\[
\widetilde{N^{\epsilon}_\bullet}:= Q^{\epsilon}_\bullet d e^{\frac{t-\bullet}{2}\Delta_{\Psi^{\epsilon}}}f(X^\epsilon_\bullet ){\ell}_\bullet -
e^{\frac{t-\bullet}{2}\Delta_{{\Psi^{\epsilon}}}}f(X^\epsilon_\bullet )\int_0^{\bullet}\big( (Q^{\epsilon}_s)^{\mathrm{tr}}\dot{\ell}_s,dW_s\big)\]
is a local martingale in $[0,t]$, with $(Q^{\epsilon}_s)^{\mathrm{tr}}$ the transpose of $Q^{\epsilon}_s$. Set now $\ell:=t^{-1}(t-\bullet)v$. We are going to show in a moment that there exists a constant $C$ such that for all $t'>0$ one has  
\begin{align}\label{aux}
 \sup_{\epsilon' > 0 } \sup_{x' \in \mathbb R^{3n}}  E^{x'}_{\Psi^{\epsilon'}}\left[e^{\int_0^{t'} k^{\epsilon'}(X_r) dr} \right] \leq e^{C t' }. 
\end{align}
It then follows from the Burkholder-Davis-Gundy inequality that the process $\widetilde{N^{\epsilon}}$ is actually a true martingale in $[0,t]$ and thus has a constant expectation value, and by passing to expectations we obtain the Bismut derivative formula  
\begin{align*}
d e^{\frac{t}{2}\Delta_{\Psi^{\epsilon}}}f(x)v=E\left[{\widetilde{N^{\epsilon}_0}}\right]= E\left[{\widetilde{N^{\epsilon}_t}}\right]= -E\left[f(X^{\epsilon}_t)\int^{t}_0\big(  (Q^{\epsilon}_s)^{\mathrm{tr}}\dot{\ell}_s,dW_s\big)\right].\end{align*}
By the approximation procedure from the proof of Theorem 1.3 iii) in \cite{braun}, the latter formula extends to $f\in L^{\infty}(\IR^{3n})$. It then follows from the Burkholder-Davis-Gundy inequality that
\[
 d e^{\frac{t}{2}\Delta_{\Psi^{\epsilon}}}f(x)v \leq \frac 1 {\sqrt t } |v| \|f\|_\infty \left(E^x_{\Psi^{\epsilon}}\left[e^{\int_0^t k^\epsilon(X^\epsilon_s) ds }\right]\right)^{\frac 1 2},
\]
which together with (\ref{aux}) implies
\[ 
| e^{\frac{t}{2}\Delta_{\Psi^{\epsilon}}}f(x) - e^{\frac{t}{2}\Delta_{\Psi^{\epsilon}}}f(x)| \leq \frac 1 {\sqrt t } e^{Ct} \|f\|_\infty |x-y|.
\]

Since the generator of the semigroup $e^{\frac{t}{2}\Delta_{\Psi}}$ is essentially self-adjoint on $L^2(\IR^{3n})$ an application of e.g.\  \cite[theorem 1.1.]{MR2130907} shows that  for  vanishing $\epsilon$ the family of semigroups $e^{\frac{t}{2}\Delta_{\Psi^\epsilon}}$ converges to $e^{\frac{t}{2}\Delta_{\Psi}}$ in the sense of Mosco. By uniform continuity and compactness in $C^{0,1}_{\textrm{loc}}(\mathbb R^{3n})$ we can thus also pass to the limit in the previous inequality giving the claimed Lipschitz estimate.  

\smallskip

It remains to show that we find a constant $C$ such that for all $t>0$ one has  

\[ \sup_{\epsilon > 0 } \sup_{x \in \mathbb R^{3n}}  E^x_{\Psi^{\epsilon}}\left[e^{\int_0^t k^\epsilon(X^\epsilon_s) ds} \right] \leq e^{C t }. \]

To see this we repeat the estimates as conducted in b) with the functions $k = |\bullet-\Sigma|^{1}$ and $\Psi$ replaced by $k^\epsilon$ and $\Psi^\epsilon$ respectively, which for given $q>1$ yields a constant $C= C(\|\nabla \Psi\|_\infty,q)$ such that one gets the first inequality in 
 \begin{align*}
 e^{\frac{t}{2}\Delta_{\Psi^{\epsilon}}} k^\epsilon (x)  & \leq e^{Ct} \left(e^{\frac{t}{2}\Delta} [k^\epsilon]^q (x)\right)^\frac 1 q \\
 & \leq 
e^{Ct} \left(e^{\frac{t}{2}\Delta} (k^q)^\epsilon (x)\right)^\frac 1 q\\
& =  
e^{Ct} \left(e^{\frac{t+\epsilon}{2}\Delta} (k^q) (x)\right)^\frac 1 q 
\end{align*}
where we have used Jen\ss{}en's inequality in the the second step. As a consequence, for $t \leq 1 $ we find
\begin{align*}
 \int_0^t   e^{\frac{s}{2}\Delta_{\Psi^{\epsilon}}} k^\epsilon (x) ds & \leq 
 e^{C} \int_0^t \left(e^{\frac{s+\epsilon}{2}\Delta} (k^q) (x)\right)^\frac 1 q ds \\
 & \leq 
 e^{C} \left(\int_0^t e^{\frac{s+\epsilon}{2}\Delta} (k^q) (x)ds\right)^\frac 1 q \\
 & = e^{C} \left(e^{\frac{\epsilon}{2}\Delta} \left[\int_0^t e^{\frac{s}{2}\Delta} (k^q)ds\right](x)\right)^\frac 1 q. 
\end{align*}

From b) we know there exists some $q>1$ such that for small enough $t$ the function $\mathbb R^{3n} \ni y \mapsto \int_0^t e^{\frac{s}{2}\Delta} (k^q)(y)ds$ can be made arbitrarily small uniformly in $y$, hence also its mollified version appearing in the last expression above. In particular, for any $0<\alpha <1$ there is some $0<t_0 <1$  for which    
\[ \sup_{t \leq t_0} \sup_{x\in \mathbb R^{3n}} \int_0^t   e^{\frac{s}{2}\Delta_{\Psi^{\epsilon}}} k^\epsilon (x) ds\leq \alpha, \]
uniformly in $\epsilon >0$. Combining the Markov property of $X^\epsilon$ with the Khasminski lemma (as in e.g.\ \cite[Lemma A.5]{MR3735415}) we finally obtain
\[ 
 E^x_{\Psi^{\epsilon}}\left[e^{\int_0^t k^\epsilon(X^\epsilon_s) ds} \right] \leq \left(\frac{1}{1-\alpha}\right)^{\frac{t}{t_0}}
\]
uniformly in $x \in \mathbb R^{3n}$, $\epsilon >0$. 
\end{proof}

As a corollary to the stochastic completeness of $\IMM_{\Psi}$ and the Feynman-Kac formula we get the following seemingly completely new formula:

\begin{Corollary} For all $t>0$, $x\in\IR^{3n}$ one has
\begin{align}\label{ende}
 E^x\left[e^{-\int^t_0V(X_s)ds}/\tilde{\Psi}(X_t)\right]=e^{-t\lambda}/\tilde{\Psi}(x).
\end{align}
\end{Corollary}

\begin{proof} Let $(K_n)$ be a compact exhaustion of $\IR^{3n}$. Then for $f_n:=1_{K_n}e^{-\Psi}\in L^2(\IR^{3n})$ we have
\begin{align}\label{ende2}
e^{-t\lambda}e^{-\Psi(x)}e^{t \Delta_{\Psi}/2}(e^{\Psi}f_n)(x)=e^{-t (-\Delta/2+V)}f_n(x)= E^x\left[e^{-\int^t_0V(X_s)ds}f_n(X_t)\right],
\end{align}
where the first equality follows from the unitary equivalence from Remark \ref{swsa33} and the second from the Feynman-Kac formula \cite{Simon}. Taking $n\to\infty$ the RHS of (\ref{ende2}) tends to the LHS of (\ref{ende}) by monotone convergence. The LHS of (\ref{ende2}) is equal to
$$
e^{-t\lambda}e^{-\Psi(x)}e^{t \Delta_{\Psi}/2}(e^{\Psi}f_n)(x)=\frac{e^{-t\lambda}}{\tilde{\Psi}(x)}\int_{K_n} e^{t \Delta_{\Psi}/2}(x,y)d\mu_{\Psi}(y),
$$
which tends to $e^{-t\lambda}/\tilde{\Psi}(x)$ by monotone convergence and the stochastic completeness of $\IMM_{\Psi}$.

\end{proof}

\textbf{Acknowledgements:} The authors would like to thank Volker Bach, Sergio Cacciatori, Thomas Hoffmann-Ostenhof and Barry Simon for very valuable discussions.



\end{document}